\documentclass[a4paper,11pt]{article}

  \renewcommand\appendix{\par
  \setcounter{section}{0}
  \setcounter{sub
  section}{0}
  \setcounter{figure}{0}
  \setcounter{table}{0}
  \renewcommand\thesection{ Appendix \Alph{section}}
  \renewcommand\thefigure{\Alph{section}\arabic{figure}}
  \renewcommand\thetable{\Alph{section}\arabic{table}}
}

 \usepackage{amsmath, amsthm, amssymb}
\usepackage{amssymb}
\usepackage{graphicx}
\usepackage{algorithmic}
\usepackage{algorithm}

\usepackage{tikz}
\usetikzlibrary{calc}
\usetikzlibrary{patterns}
\tikzstyle{mybox} = [draw=black, fill=white,  thick,
    rectangle, inner sep=10pt, inner ysep=20pt]
\tikzstyle{mybox} = [draw=black, fill=white,  thick,
    rectangle, inner sep=2pt, inner ysep=2pt]

\usepackage[margin=1.2in]{geometry}

\usepackage{setspace}
\newtheorem{lemma}{Lemma}[section]
\newtheorem{theorem}{Theorem}[section]

\begin{document}

\title{On $k$-tuple and $k$-tuple total domination numbers of regular graphs}

%% use optional labels to link authors explicitly to addresses:
%% \author[label1,label2]{}
%% \address[label1]{}
%% \address[label2]{}

\author{Sharareh Alipour \and Amir Jafari \and Morteza Saghafian}

\maketitle
\begin{abstract}
Let $G$ be a connected graph of order $n$, whose minimum vertex degree is at least $k$. A subset $S$ of vertices in $G$ is a $k$-tuple total dominating set if every vertex of $G$ is adjacent to at least $k$ vertices in $S$. The minimum cardinality of a $k$-tuple total dominating set of $G$ is the $k$-tuple total domination number of $G$, denoted by  $\gamma_{\times k,t}(G)$. Henning and Yeo in \cite{hen} proved that if $G$ is a cubic graph different from the Heawood graph, $\gamma_{\times 2, t}(G) \leq \frac{5}{6}n$, and this bound is sharp.
Similarly, a $k$-tuple dominating set is a subset $S$ of vertices of $G$, $V (G)$ such that $|N[v] \cap S| \geq k$ for every vertex $v$, where 
$N[v] = \{v\}\cup \{u \in V(G) : uv \in E(G)\}$. The $k$-tuple domination number of $G$, denoted by $\gamma_{\times k}(G)$, is the minimum cardinality of a $k$-tuple dominating set of $G$. 

In this paper, we give a simple approach to compute an upper bound for $(r-1)$-tuple total domination number of  $r$-regular graphs. Also, we give an upper bound for the $r$-tuple dominating number of $r$-regular graphs.
In addition, our method gives algorithms to compute dominating sets with the given bounds, while the previous methods are existential.
\end{abstract}

%% \linenumbers

%% main text
\section{Introduction}
\subsection{Problem statement}
In a graph $G$ with vertex set $V(G)$ and edge set $E(G)$, the open neighborhood of a vertex $v$ is $N(v) = \{u \in V(G):uv \in E(G)\}$ and the closed neighborhood is $N [v] = \{v\} \cup N (v)$. The degree of $v$, denoted by $d(v)$, is the cardinality of $N(v)$. In a graph $G$, $\delta(G)$ is  the minimum degree of vertices of $G$ and $G$ is $r$-regular if $d(v) = r$ for all $v \in V$ .

For a positive integer $k$, a $k$-tuple total dominating set of $G$ is a subset $S$ of $V(G)$ such that $|N(v)\cap S| \geq k$ for all $v \in V(G)$. 
Also, a $k$-tuple dominating set is a subset $S$ of $V(G)$ such that $|N[v]\cap S| \geq k$ for all $v \in V(G)$. 
In the case of $k=1$, $1$-tuple total dominating set and $1$-tuple dominating set are simply called total dominating set and dominating set, respectively. 

The minimum cardinalities of $k$-tuple total dominating sets and $k$-tuple dominating sets are denoted by $\gamma_{\times k,t}(G)$ and $\gamma_{\times k}(G)$, respectively. 
\subsection{Related works and our results}

Domination in graphs is now well studied in graph theory and the literature on this subject has been surveyed and detailed in the two books by Haynes, Hedetniemi, and Slater \cite{hed,sla}.
Dominating sets are of practical interest in several areas. In wireless networking, dominating sets are used to find efficient routes within ad-hoc mobile networks. They have also been used in document summarization, and in designing secure systems for electrical grids. Also, the concept of domination in graph theory is a natural model for many location problems in operations research.
A main application to network purposes of $k$-tuple domination is for fault tolerance or mobility in the following situations. Each vertex of the graph models a node of the network and edges are links. Node $u$ can use a service (any read-only database for example) only if it is replicated on $u$ or on a neighbor of $u$. To ensure a certain degree of fault tolerance or to tolerate mobility of nodes, one can imagine that any node $u$ has in its (closed) neighborhood at least $k$ copies of this service available. As each copy can cost a lot, the number of duplicated copies has to be minimized \cite{klas}.

The complexity of the domination problem has been well-studied in the literature, see \cite{chang}. The hardness of approximation of the  domination problem has also been extensively investigated in the literature, see \cite{aus}. In terms of the complexity of the $k$-tuple domination problem in graphs, a linear-time algorithm for the $2$-tuple domination problem in trees is given in \cite{lia}. A linear-time algorithm for the $k$-tuple domination problem in strongly chordal graphs is presented in \cite{lia2}, where it is also proved that $k$-tuple domination is NP-complete for split graphs and for bipartite graphs.
In \cite{klas}, Klasing and Laforest described a $(\ln |V | + 1)$-approximation algorithm for the $k$-tuple domination problem in
general graphs, and showed that $k$-tuple domination cannot be approximated within a ratio
of $(1-\epsilon)\ln |V|$ for any $\epsilon > 0$ unless $NP\subset DTIME(|V|O(\log \log |V|))$. Then, they proved that the $k$-tuple domination problem can be approximated within a constant ratio if the degree of the graph is bounded by a constant, but that it is APX-hard to approximate for graphs of maximum degree $k + 2$. Also, they showed that the $k$-tuple domination problem can be approximated within a constant ratio in $p$-claw free graphs, but that it is APX-hard to approximate for $5$-claw free graphs. $p$-claw free graphs are graphs which do not have $K_{1,p}$ (a star with $p$ leaves) as an induced subgraph.

While determining the exact value of $\gamma_{\times k,t}(G)$ and $\gamma_{\times k}(G)$ for a graph $G$ are not easy, many studies focus on their upper bounds\cite{har2,klas,lia,lia2}. Here, we present the known upper bounds.

Let $G_{14}$ be the Heawood graph (or, equivalently, the incidence bipartite graph of the Fano plane) on $14$ vertices shown in Figure \ref{fig1}. In \cite{hen}, Henning  and Yeo proved some theorems about strong transversal in hypergraphs and then as an application of their hypergraph results they proved the following theorem.
\begin{theorem}{\cite{hen}}
\label{1}
If $G \neq G_{14}$ is a connected cubic graph of order $n$, then $\gamma_{\times 2,t}(G) \leq \frac{5}{6}n$, and this bound is sharp and $\gamma_{\times 2,t}(G_{14})=12$.
\end{theorem}

\begin{figure}[h!]
\centering
  \includegraphics[width=50mm]{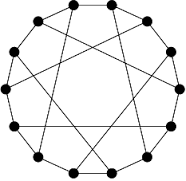}
  \caption{The Heawood graph, $G_{14}$.} 
 \label{fig1}
\end{figure}

Now, let, $\tilde{d}_m=\frac{1}{n} \sum^n_{i=1} {d(v_i)+1\choose m}$. Then, we have the following theorem from \cite{chang2}.
\begin{theorem}{\cite{chang2}}
For any graph $G$ of minimum degree $\delta$ with $1 \le k\le \delta+ 1$

$$\gamma_{\times k}(G)\leq \frac{\ln(\delta-k+2)+\ln \tilde{d}_{k-1}+1}{\delta-k+2} n.$$
\label{11}
\end{theorem}

Also, let, $\hat{d}_m=\frac{1}{n} \sum^n_{i=1} {d_i\choose m}$, then we have the following theorem from \cite{kaz}.
\begin{theorem}{\cite{kaz}}
If $k$ is a positive integer and $G$ is a graph of order $n$ with $\delta > k \geq 1$, then
$$\gamma_{\times k,t}(G)\leq \frac{\ln(\delta-k)+\ln \hat{d}_{k}+1}{\delta-k} n.$$
\label{22}
\end{theorem}

In the next two sections, we give upper bounds for $(r-1)$-tuple total domination number and $r$-tuple domination number of $r$-regular graphs. Our result for $(r-1)$-tuple total domination number of $r$-regular graphs is an extension of Theorem \ref{1}. Theorem \ref{1} computes an upper bound for $2$-tuple total domination number of $3$-regular graphs, note that for $r=3$, our theorem implies Theorem \ref{1}.

\section{$r$-tuple total dominating set}
In this section, we give a theorem for $(r-1)$-tuple total domination number of $r$-regular graphs.
\begin{theorem}
\label{2}
Let $r\ge 3$. If $G$ is an $r$-regular graph of order $n$ which is not the incident graph of a projective plane of order $r-1$, then $\gamma_{\times (r-1),t}(G) \leq \frac{r(r-1)-1}{r(r-1)}n$. If $G$ is the incidence graph of a projective plane of order $r-1$ then, $\gamma_{\times (r-1),t}(G)=\frac{r(r-1)}{r(r-1)+1}n=2r(r-1)$. 
\end{theorem}

\begin{proof}
First we give the main idea of the proof. We construct a graph $G'$ from $G$ as follows. The vertices of $G'$ are the same as $G$ and two vertices are joined by an edge in $G'$ if and only if they have a common neighbor in $G$. We show that a proper coloring of $G'$, help us construct an $(r-1)$-tuple total dominating set for $G$. 

Now we show that a proper coloring of vertices of $G'$, gives a coloring for $G$ such that the colors of neighbors of each vertex of $G$ are different. Suppose that we use $c$ colors for this coloring, then there are at least $\frac{n}{c}$ of the vertices, that have the same color, suppose color red. So, the number of vertices with a color other than red, is at most $(1-\frac{1}{c})n$.
We claim that the set of vertices with a color other than red, denoted by $D$ is an $(r-1)$-tuple total dominating set for this graph. Since the colors of neighbors of each vertex are different, so every vertex has at least $r-1$ colors different than red and so $D$ is an $(r-1)$-tuple total dominating set.
Therefore, our aim is to minimize $c$.

Since $G$ is an $r$-regular graph, so $\Delta(G')\le r(r-1)$.
By Brooks' theorem, if $G'$ does not have a complete component of order $r(r-1)+1$, then $G$ has a proper coloring with $r(r-1)$ colors.
Now, we show that if $G'$ has a complete component of order $r(r-1)+1$ then $G$ is the incidence graph of a projective plane of order $r-1$. Note that $G'$ can not be a complete graph of order $r(r-1)+1$, because there is no $r$-regular graph of order $r(r-1)+1$ that any two of its vertices have a common neighbor. So $G'$ is disconnected and hence $G$ is bipartite. Because any two vertices in $G$ that have a path of even length, will have a path in $G'$ of half that length. Hence, if $G'$ is not connected, there exist two vertices with no even paths between them, this implies that $G$ has no odd cycles, because otherwise by connectedness of $G$, there is a path between these two vertices that has a vertex from this odd cycle, so by adding this cycle to this path, we get a path of even length. Since $G$ is $r$-regular, so each part of the bipartite graph $G$ have the same number of vertices. Since between every two vertices of each part there are paths of even length, so the vertices of each part form a connected component of $G'$. Hence, each part has $r(r-1)+1$ vertices. 
Also, in one part of G, each pair of vertices has a unique common neighbor. So this property holds in the other part, because the number of vertices and the degrees are the same in both parts. This is because, every two vertices in one part of $G$ are joined by an edge in $G'$ and hence they have a common neighbor. The number of these neighbors is at most ${{r(r-1)+1}\choose {2}}/ {r\choose 2}$ which is $r(r-1)+1$. This is exactly the number of vertices of the other part, so there can not be two vertices with more than one common neighbor.
Therefore $G$ is the incidence graph of the finite projective plane of order $r-1$.

Now we show that the $(r-1)$-tuple dominating number of the incidence graph of a projective plane of order $r-1$ is exactly $2r(r-1)$. Assume we have a set of size less than $2r(r-1)$, then at least the size of this set in one of the two parts of the bipartite graph $G$ is less than $r(r-1)$. So, it misses at least two vetrices from that part, since these two vertices have a common neighbor in the other part and that vertex has degree $r$, hence it has at most $r-2$ neighbors from our set, showing that this set is not $(r-1)$-tuple dominating set.
\end{proof}

\section{$r$-tuple dominating set}
Similar to the previous section, we give a theorem about the $\gamma_{\times r}(G)$ for $r$-regular graphs. 
First, we need to present the definition of a special class of graphs known as Moore graphs. A Moore graph is a regular graph of degree $r$ and diameter $d$ whose number of vertices equals to the upper bound
$$
1+r\sum _{i=0}^{d-1}(r-1)^{i}.
$$.
\begin{theorem}
\label{3}
If $G$ is an $r$-regular graph of order $n$ which is not a Moore graph of degree $r$ and diameter $2$ then $\gamma_{\times r}(G) \leq \frac{r^2-1}{r^2}n$, otherwise $\gamma_{\times r}(G)=\frac{r^2}{r^2+1}n=r^2$.
\end{theorem}
\begin{proof}
The idea is similar to previous section, we want to color the vertices of $G$ such that the color of each vertex is different from its neighbors and also the colors of neighbors of each vertex are also different. So, we construct a graph $G''$ such that if $v_i$ and $v_j$ are adjacent in $G$ then $v_iv_j\in E(G'')$. Also, if $v_i$ and $v_j$ are adjacent to $v_k$ in $G$, then $v_iv_j\in E(G'')$.  Similar to the previous section, we color the vertices of $G''$ properly. Note that $\Delta(G'')\leq r^2$.
Now, we show that the only case that $G''$ has a complete component of order $r^2+1$ is when $G= G_{2,r}$. $G''$ is connected because $G$ is connected and it is a subgraph of $G''$. So, $G''$ should be a complete graph of order $r^2+1$ and so the diameter of $G$ should be $2$. which means, $G$ is a Moore graph of degree $r$ and diameter $2$. It is proven that except for $r=2,3,7,57$, Moore graph of diameter $2$ can not exist. Moore graphs of order $r=2,3,7$ is known, but for $r=57$ it is unknown \cite{moore}. 
So, if $G$ is not a Moore graph, then we need $r^2$ colors for coloring of $G''$. The last part follows from the following lemma.
\end{proof}
\begin{lemma} 
If $G$ is an $r$-regular graph of diameter $2$ with $n$ vertices then $\gamma_{\times r}(G)=n-1$.
\end{lemma}
\begin{proof}
It is clear that any $n-1$ vertices of $G$ form an $r$-tuple dominating set. Now assume $S$ is a subset of vertices of size at most $n-2$, then there are two vertices $x$ and $y$ that are not in $S$. If $x$ and $y$ are not connected then since the diameter is $2$, they have a common neighbor $z$, which can have at most $r-2$ neighbors in $S$. This shows that $S$ is not an $r$-tuple dominating set. If $x$ is connected to $y$ then $x$ can have at most $r-1$ neighbors in $S$ and therefore again $S$ is not an $r$-tuple dominating set. This proves the lemma.
\end{proof}

\section{Conclusion}
In this paper, we use a simple idea for computing upper bounds for $(r-1)$-tuple total domination number and  $r$-tuple domination number for $r$-regular graphs.
Theorem \ref{2} is a general case of theorem \ref{1}. Also, Theorem \ref{2} and Theorem \ref{3} give better upper bounds than Theorem \ref{11} and Theorem \ref{22} for $\gamma_{\times r-1,t}(G)$ and $\gamma_{\times r}(G)$ when $G$ is a $r$-regular graph, respectively.
Our proofs are simple and also algorithmic,  in contrast to the previous results that were existential.
Also, this idea is applicable for similar situation when the minimum vertex degree is bounded from below. In order to keep our paper simple, we omitted theses applications.

\end{document}